\documentclass[preprint,12pt]{elsarticle}

\usepackage{amssymb}
\usepackage{amscd}
\usepackage{amsthm}
\usepackage{amsmath}
\usepackage{latexsym}
\usepackage{graphicx}
\usepackage{delarray}
\usepackage{float}
\usepackage[dvips]{epsfig}
\usepackage{subfig}
\usepackage{hyperref}\usepackage{breakurl}
\DeclareMathOperator{\esssup}{ess\,sup}
\DeclareMathOperator{\essinf}{ess\,inf}
\begin{document}

\theoremstyle{plain}
\newtheorem{theorem}{Theorem}[section]
\newtheorem{claim}{Claim}
\newtheorem{lemma}[theorem]{Lemma}
\newtheorem{proposition}[theorem]{Proposition}
\newtheorem{corollary}[theorem]{Corollary}

\theoremstyle{remark}
\newtheorem{example}{Example}[section]
\newtheorem{remark}{Remark}
\theoremstyle{definition}
\newtheorem{definition}[theorem]{Definition}
\hfuzz5pt 


\newcommand{\gt}{\tilde{g}}
\newcommand{\R}{\mathbb{R}}
\newcommand{\Z}{\mathbb{Z}}
\newcommand{\N}{\mathbb{N}}
\newcommand{\Zt}{\mathbb{Z}^2}
\newcommand{\Zd}{\mathbb{Z}^d}
\newcommand{\Ztd}{0\mathbb{Z}^{2d}}
\newcommand{\Rt}{\R^2}
\newcommand{\Rtd}{\R^{2d}}
\newcommand{\Zp}{Z_p}
\newcommand{\al}{\alpha}
\newcommand{\be}{\beta}
\newcommand{\om}{\omega}
\newcommand{\Om}{\Omega}
\newcommand{\ga}{\gamma}
\newcommand{\de}{\delta}
\newcommand{\la}{\lambda}
\newcommand{\La}{\Lambda}
\newcommand{\ala}{\la^\circ}
\newcommand{\aLa}{\La^\circ}
\newcommand{\nat}{\natural}
\newcommand{\G}{\mathcal{G}}
\newcommand{\A}{\mathcal{A}}
\newcommand{\V}{\mathcal{V}}
\newcommand{\M}{\mathcal{M}}
\newcommand{\MV}{\mathcal{MV}}
\newcommand{\MP}{\mathcal{MP}}
\newcommand{\Sp}{\mathcal{S}}
\newcommand{\W}{\mathcal{W}}
\newcommand{\Hp}{\mathcal{H}}
\newcommand{\ka}{\kappa}
\newcommand{\cast}{\circledast}
\newcommand{\id}{\mbox{Id}}
\newcommand{\by}{\mathbf{Y}}
\newcommand{\bx}{\mathbf{X}}
\newcommand{\bz}{\mathbf{Z}}
\newcommand{\bn}{\mathbf{N}}
\newcommand{\bv}{\mathbf{v}}
\newcommand{\bu}{\mathbf{u}}
\newcommand{\bA}{\mathbf{A}}
\newcommand{\bC}{\mathbf{C}}
\newcommand{\bD}{\mathbf{D}}
\newcommand{\bh}{\mathbf{h}}
\newcommand{\bff}{\mathbf{f}}
\newcommand{\bH}{\mathbf{H}}
\newcommand{\bV}{\mathbf{V}}
\newcommand{\bX}{\mathbf{x}}
\newcommand{\byy}{\mathbf{y}}
\newcommand{\bZ}{\mathbf{z}}
\newcommand{\bt}{\mathbf{t}}
\newcommand{\bs}{\mathbf{\sigma}}
\newcommand{\bI}{\mathbf{I}}
\newcommand{\Gn}{\mathcal{G}(\mathbf{g},\mathbf{b})}
\newcommand{\lo}{{\ell^1}}
\newcommand{\Lo}{{L^1}}
\newcommand{\Lt}{{L^2}}
\newcommand{\SO}{{S_0}}
\newcommand{\SOc}{{S_{0,c}}}
\newcommand{\Lp}{{L^p}}
\newcommand{\Los}{{L^1_s}}
\newcommand{\WCl}{{W(C_0,\ell^1)}}
\newcommand{\lt}{{\ell^2}}

\newcommand{\conv}[2]{{#1}\,\ast\,{#2}}
\newcommand{\twc}[2]{{#1}\,\nat\,{#2}}
\newcommand{\mconv}[2]{{#1}\,\cast\,{#2}}
\newcommand{\set}[2]{\Big\{ \, #1 \, \Big| \, #2 \, \Big\}}
\newcommand{\inner}[2]{\langle #1,#2\rangle}
\newcommand{\innerBig}[2]{\Big \langle #1,#2 \Big \rangle}
\newcommand{\dotp}[2]{ #1 \, \cdot \, #2}

\newcommand{\Zpd}{\Zp^d}
\newcommand{\I}{\mathcal{I}}
\newcommand{\J}{\mathcal{J}}
\newcommand{\Zq}{Z_q}
\newcommand{\Zqd}{Zq^d}
\newcommand{\Zak}{\mathcal{Z}_{a}}
\newcommand{\C}{\mathbb{C}}
\newcommand{\F}{\mathcal{F}}

\newcommand{\convL}[2]{{#1}\,\ast_L\,{#2}}
\newcommand{\abs}[1]{\lvert#1\rvert}
\newcommand{\absbig}[1]{\big\lvert#1\big\rvert}
\newcommand{\absBig}[1]{\Big\lvert#1\Big\rvert}
\newcommand{\scp}[1]{\langle#1\rangle}
\newcommand{\norm}[1]{\lVert#1\rVert}
\newcommand{\normmix}[1]{\lVert#1\rVert_{\ell^{2,1}}}
\newcommand{\normbig}[1]{\big\lVert#1\big\rVert}
\newcommand{\normBig}[1]{\Big\lVert#1\Big\rVert}


\begin{frontmatter}

\title{Nonstationary Gabor Frames - Existence and Construction}

\author{Monika D\"orfler}
\ead{monika.doerfler@univie.ac.at}

\author{Ewa Matusiak\corref{cor1}}
\ead{ewa.matusiak@univie.ac.at}

\cortext[cor1]{Corresponding author}

\address{Department of Mathematics, NuHAG, University of Vienna, Austria}


\begin{abstract}
Nonstationary Gabor frames were recently introduced in adaptive signal analysis. They  represent a natural generalization of classical Gabor frames by allowing for adaptivity of windows and lattice in either time or frequency. In this paper we show a general existence result for this family of frames. We then  give a perturbation result for nonstationary Gabor frames and  construct nonstationary Gabor frames with non-compactly supported windows from a related painless nonorthogonal expansion.  Finally, the  theoretical  results are illustrated by two examples of practical relevance.

\end{abstract}


\end{frontmatter}


\section{Introduction}\label{sec:intro}
The principal idea of Gabor frames was introduced in~\cite{ga46} with the aim to
represent signals in a time-frequency localized manner. Since the work of Gabor
himself, a lot of research has been done on the topic of atomic time-frequency
representation. While it turned out that the original model proposed by Gabor
does not yield stable representations in the sense of
frames~\cite{ch03,dagrme86,dusc52}, the existence of Gabor frames was first
established  in the so called painless case,~\cite{dagrme86}, which requires the
use of compactly supported analysis windows. 
The existence of Gabor frames in more general situations was
proved later~\cite{rosh97,wa92} and the proof often uses an argument invoking
the invertibility of diagonally dominant matrices.

Various irregular and adaptive versions of Gabor frames have been introduced over the years, cf.~\cite{alcamo04-1,chfazo01,faza95,suzh02}.  In these approaches, the irregularity usually concerns either the sampling set, which is allowed to deviate from a lattice, or the window, which is allowed to be modified. In~\cite{alcamo04-1}, varying windows as well as irregular sampling points are allowed, however, existence of a local frame is assumed, from which a global frame is constructed. \\
 Nonstationary Gabor frames give up the
strict regularity of the classical Gabor setting, but, as opposed to irregular
frames,  maintain enough structure to guarantee  efficient implementation
and,  possibly approximate, efficient reconstruction. In analogy to the
classical, regular case~\cite{dagrme86}, painless nonstationary Gabor frames
were introduced in~\cite{badohojave11}, where the principal idea is  described and
illustrated in detail. The construction of painless nonstationary  Gabor frames  is similar to, but more flexible than  the construction of
windowed modified  cosine transforms and other lapped transforms~\cite{ma92-2,maxi01} that allow for
adaptivity of the window length.  An efficient and perfectly invertible constant-Q transform was recently
introduced using nonstationary Gabor transforms~\cite{dohove11}. In this and similar situations, 
redundancy of the transform is crucial,  since non-redundant versions of the constant-Q transform lead to dyadic wavelet
transforms, which are often inappropriate  for audio signal processing. \\
Redundancy allows for good localization of
both the analysis and synthesis windows, and their respective Fourier
transforms and often promote sparse representations in
adaptive processing.

Painless non-orthogonal expansions can only be devised if the involved analysis windows are  either
compactly supported or band limited. This requirement may sometimes be too restrictive. For instance, one may be interested in designing  frequency-adaptive nonstationary Gabor frames with windows that are compactly supported in time, i.e. can be implemented as FIR filters,~cf.~\cite{evdoma12} and lend themselves to  real-time implementation,~cp.~\cite{dogrhove12}.

The present contribution addresses the case of nonstationary Gabor frames with more general windows than used in the painless case.
In Theorem~\ref{thm:nonstationary_frames_1}, we derive the existence of nonstationary Gabor frames directly from a
generalized Walnut representation: under mild uniform decay conditions on all windows involved, we show an
existence result of nonstationary Gabor frames in parallel to the result given
in~\cite{wa92} for regular Gabor frames, also
compare~\cite[Theorem~6.5.1]{gr01}. Note that the existence of a different class of nonstationary Gabor frames, the {\it quilted Gabor frames},~\cite{do11}, was recently proved in the general context of spline type spaces in the remarkable paper~\cite{ro11}.
 
This paper is organized as follows. In the next section, we introduce notation and state some auxiliary results. In Section~\ref{MainSec}, we first define nonstationary Gabor frames and recall known results for the painless case. In Section~\ref{WalnutSec} a Walnut representation and a corresponding bound of the frame operator in the general setting is derived and Section~\ref{Se:Exis} provides the existence  of nonstationary Gabor frames. 
In Section~\ref{Se:apnGF} we pursue two  basic approaches for  the construction of nonstationary Gabor frames .
 Using tools from the theory of  perturbation of  frames, we construct
nonstationary frames from  an existing  frame in Proposition~\ref{thm:nonstationary_frames}.
In Corollary~\ref{cor:nonstationary_frames_2} we design  nonstationary Gabor frames by exploiting knowledge about a related painless  frame, to obtain  "almost painless nonstationary Gabor frames". 
 In Section~\ref{Se:Ex} we provide examples based  on the two introduced construction principles.

\section{Notation and Preliminaries}

Given a non-zero function $g\in\Lt(\R)$, let $g_{k,l}(t)= M_{bl} T_{ak} g(t) := e^{2\pi i bl t} g(t-ak)$. $M_{bl}$ is a modulation operator, or frequency shift, and $T_{ak}$ is a time shift.

The set $\mathcal{G}(g,a,b) = \{g_{k,l}\,:\, k,l\in\Z \}$ is called a Gabor system for any real, positive $a,b$. $\mathcal{G}(g,a,b)$ is a Gabor frame for $\Lt(\R)$, if there exist frame bounds $0<A \leq B < \infty$ such that for every $f\in\Lt(\R)$ we have
\begin{equation}\label{Eq:framecond}
A \norm{f}_2^2 \leq \sum_{k,l\in\Z} \abs{\inner{f}{g_{k,l}}}^2 \leq B \norm{f}_2^2\,.
\end{equation}

To every Gabor system,  we associate the analysis operator $C_g$ given by $(C_g f)_{k,l} = \inner{f}{g_{k,l}}$, and the synthesis operator $U_\ga=C_\ga^\ast$, given by $U_\ga c = \sum_{k,l\in\Z} c_{k,l} \ga_{k,l}$ for $c\in\ell^2$. The operator $S_{g,\ga}$ associated to $\mathcal{G}(g,a,b)$ and $\mathcal{G}(\ga,a',b')$, where $S_{g,\ga}=U_\ga C_g$ reads
\begin{equation*}
S_{g,\ga} f = \sum_{k,l\in\Z} \inner{f}{g_{k,l}} \ga_{k,l}\,. 
\end{equation*}
The inequality \eqref{Eq:framecond} is equivalent to the invertibility and boundedness of the frame operator $S_{g,g}$ of $\mathcal{G}(g,a,b)$.

The analysis operator is the sampled short-time Fourier transform (STFT). For a fixed window $g\in\Lt(\R)$, the STFT of $f\in\Lt(\R)$ is
\begin{equation*}
V_g f (x,\om) = \int_{\R} f(t) e^{-2\pi i \om t} \overline{g(t-x)}\,dt = 
\inner{f}{M_\om T_x g}\,.
\end{equation*}
Setting  $(x,\om) = (ak,bl)$, leads to $V_g f (ak,bl) = (C_g f)_{k,l}$. 

When working with irregular grids,  we assume that the sampling points form a separated set: a set of sampling points $\{a_k\,:\, k\in\mathbb{Z}\}$ is called $\delta$-separated, if  $\abs{a_k-a_m} > \delta$ for 
 $a_k$, $a_m$, whenever $k\neq m$.
$\chi_{I}$ will denote the characteristic function of the interval $I$.

A convenient class of window functions for time-frequency analysis on $\Lt(\R)$ is the Wiener space.
\begin{definition}
A function $g\in L^{\infty}(\R)$ belongs to the Wiener space
$W(L^{\infty},\ell^1)$ if
\begin{equation*}
\norm{g}_{W(L^{\infty},\ell^1)} := \sum_{k\in\Z}
\mbox{ess sup}_{t\in Q} \abs{g(t+k)} < \infty\,,\quad Q=[0,1]\,.
\end{equation*}
\end{definition}

For $g\in W(L^{\infty},\ell^1)$ and $\delta > 0$ we have \cite{gr01}
\begin{equation}\label{eq:wiener_norm}
\esssup_{t\in\R} \sum_{k\in\Z} \abs{g(t-\delta k)} \leq (1+\delta^{-1})
\norm{g}_{W(L^{\infty},\ell^1)}\,.
\end{equation}

In dealing with polynomially decaying windows, we will repeatedly use the
following lemma.

\begin{lemma}\label{lem:estimates}
For $p>1$ the following estimates hold: 
\begin{itemize}
\item[(a)] Let $\delta >0$, then
\begin{equation*}
\sum_{k=1}^{\infty} (1+ \delta k)^{-p}
\leq (1+\delta)^{-p}(\delta^{-1}+ p)(p-1)^{-1}\,.
\end{equation*}
\item[(b)] Let $\{ a_k\,:\,k\in\Z \} \subset \R$ be
a $\delta-$separated set. Then
\begin{equation*}
\esssup_{t\in\R} \sum_{k\in\Z} (1+\abs{t-a_k})^{-p} \leq 2\left ( 1 +
(1+\delta)^{-p}(\delta^{-1}+ p)(p-1)^{-1} \right )\,.
\end{equation*}
\end{itemize}
\end{lemma}

\begin{proof}
To show (a) we write
\begin{equation*}
\sum_{k=1}^{\infty} (1+\delta k)^{-p} = (1+\delta)^{-p} + \sum_{k=2}^{\infty}
(1+\delta k)^{-p} = (1+\delta)^{-p} + \sum_{k=2}^{\infty} \int_{[0,1]+k}
(1+\delta k)^{-p} \,dt\,.
\end{equation*}
for $t\in [k,k+1]$, we have $\delta t \leq \delta (k+1)$ which implies that
$1+\delta(t-1) \leq 1+ \delta k$. Therefore,
\begin{align*}
\sum_{k=2}^{\infty} \int_{[0,1]+k} (1+\delta k)^{-p} \,dt
&\leq \sum_{k=2}^{\infty} \int_{[0,1]+k} (1+\delta (t-1))^{-p} \, dt=
\int_{2}^{\infty} (1+\delta (t-1))^{-p}\,dt \nonumber \\
&= (1+\delta)^{-p+1} \delta^{-1}(p-1)^{-1}\,,
\end{align*}
and the estimate follows.

To prove (b), fix $t\in\R$. Since $\abs{a_k - a_l} > \delta$ for
$k\neq l$, each interval of length $\delta$ contains at most one point
$t-a_k$, $k\in\Z$. Therefore we may write $t-a_k = \delta n_k + x_k$ for unique
$n_k\in\Z$ and $x_k\in[0,\delta)$, and by the choice of $\delta$, we have
$n_k\neq n_l$ for $k\neq l$. We assume, without loss of generality, that
$n_k=0$ for $k=0$, and we find that
\begin{align*}
\sum_{k\in\Z} (1+\abs{t-a_k})^{-p} &= \sum_{k\in\Z} (1+\abs{\delta n_k +
x_k})^{-p} \nonumber \\
&\leq 1 + \sum_{k\in\Z\,;\, n_k > 0} (1+\delta n_k +x_k)^{-p} +
\sum_{k\in\Z\,;\, n_k > 0} (1+\delta n_k - x_k)^{-p} \nonumber \\
&\leq 1 + \sum_{k\in\Z\,;\, n_k > 0} (1+\delta n_k)^{-p} + \sum_{k\in\Z\,;\, n_k
> 0} (1+\delta n_k - \delta)^{-p} \nonumber \\
&\leq 1+ \sum_{k=1}^{\infty} (1+\delta k)^{-p} + \sum_{k=1}^{\infty}
(1+\delta (k-1))^{-p} \nonumber \\
&= 2 \left ( 1 + \sum_{k=1}^{\infty} (1+\delta k)^{-p} \right ) \leq
2\left ( 1 + (1+\delta)^{-p}(\delta^{-1}+ p)(p-1)^{-1} \right )\,.
\end{align*}
The last expression is independent of $t$, and the claim follows.
\end{proof}

\begin{remark}\label{rem:estimate}
When the set $\mathcal{A} = \{ a_k\,:\,k\in\Z \} \subset \R$ is relatively $\delta-$separated, meaning
\begin{equation*}
\mbox{rel}(\mathcal{A}) := \max_{t\in\R}\,\# \{ \mathcal{A} \cap ([0,\delta] + t)\} < \infty\,,
\end{equation*}
then the estimate (b) in Lemma~\ref{lem:estimates} becomes
\begin{equation*}
\esssup_{t\in\R} \sum_{k\in\Z} (1+\abs{t-a_k})^{-p} \leq 2\, \mbox{rel}(\mathcal{A})\, \left ( 1 + (1+\delta)^{-p}(\delta^{-1}+ p)(p-1)^{-1} \right )\,.
\end{equation*}
Notice, that for a separated set, $\mbox{rel}(\mathcal{A}) = 1$.
\end{remark}

\section{Nonstationary Gabor frames}\label{MainSec}

Nonstationary Gabor systems provide a  
generalization of the classical Gabor systems of time-frequency-shifted versions of a single window function.
\begin{definition}
Let  $\mathbf{g}=\{g_k\in\Lt(\R): \;k\in\Z\}$ be  a set of window
functions  and let $\mathbf{b} = \{b_k: \;k\in\Z\}$  be a corresponding sequence of frequency-shift
parameters. Set $g_{k,l}  = M_{b_k l} g_k$. Then, the set 
\begin{equation*}
\mathcal{G}(\mathbf{g},\mathbf{b})=\{g_{k,l}:\; k,l\in\mathbb{Z}\}
\end{equation*}
is called a {\it nonstationary Gabor system}.
\end{definition}
Note that, conceptually, we assume that the windows $g_k$ are centered at points $\{a_k\, :\, k\in \Z \}$, in direct generalization of the regular case, where $g_k ( t) = g(t-ak)$ for some time-shift parameter $a$. In this sense, we have a two-fold generalization: the sampling points can be irregular {\it and} the windows can change for every sampling point. 
We are interested in 
conditions under which a nonstationary Gabor system forms a frame. We first recall the case of
nonstationary Gabor frames with compactly supported windows,
see~\cite{badohojave11} and \url{http://www.univie.ac.at/nonstatgab/} for
further information.

\subsection{Compactly supported windows: the painless case}\label{subsec:compactly}

Based on the  support length of the windows $g_k$, we can easily determine 
frequency-shifts parameters $b_k$, for which we obtain a frame. The following
result is the nonstationary version of the result given in~\cite{dagrme86}.

\begin{proposition}[\cite{badohojave11}]\label{BDF}
Let $\mathbf{g}=\{g_k\in\Lt(\R): \;k\in\Z\}$ be a collection of compactly supported functions with
$\abs{\mbox{supp }g_k} \leq 1/b_k$. Then $\mathcal{G}(\mathbf{g},\mathbf{b})$ is
a frame for $\Lt(\R)$ if there exist constants $A>0$
and $B<\infty$ such that
\begin{equation*}
A \leq G_0(t) = \sum_{k\in \Z} b_k^{-1} \abs{g_k(t)}^2 \leq B\,\mbox{ a.e. }.
\end{equation*}
The dual atoms are then $\ga_{k,l}(t) = M_{lb_k} G_0^{-1}(t) g_k(t)$.
\end{proposition}


\begin{remark}
An analogous theorem holds for bandlimited functions $g_k$.
\end{remark}
The above theorem follows from the fact that the frame operator associated to
the collection of atoms described in the theorem can be written as
\begin{equation*}
S_{g,g}f(t) = \sum_{k\in\Z} b_k^{-1} \abs{g_k(t)}^2 f(t)\,\mbox{ a.e. }.
\end{equation*}
The diagonality of the frame operator in the painless case is derived from a
generalized Walnut representation for the frame operator $S_{g,g}$ of  nonstationary
Gabor frames. In the next section we will see that this representation immediately implies diagonality of $S_{g,g}$ under the assumptions of Proposition~\ref{BDF}.


\subsection{A Walnut representation for nonstationary Gabor Frames}~\label{WalnutSec}

Let us now consider nonstationary Gabor systems $\Gn$ and $\mathcal{G}(\mathbf{\ga},\mathbf{b})$, with  all  windows $g_k$ and $\ga_k$
in $W(L^\infty, \ell^1 )$. The operator associated to $\Gn$ and $\mathcal{G}(\mathbf{\ga},\mathbf{b})$ reads
\begin{equation}\label{Eq:FOnon}
S_{g,\ga}f = \sum_{k,l\in\Z} \inner{f}{M_{lb_k}g_k} M_{lb_k}\ga_k\,.
\end{equation}

\begin{proposition}\label{prop:non_walnut}
The operator $S_{g,\ga}$ in  \eqref{Eq:FOnon} admits a Walnut representation
\begin{equation}\label{eq:walnut_general}
S_{g,\ga}f = \sum_{k,l\in\Z} G^{g,\gamma}_{k,l} \cdot T_{lb_k^{-1}}f\,,\quad \mbox{where} \quad
G^{g,\gamma}_{k,l}(t) = b_k^{-1} \overline{g_k(t-lb_k^{-1})} \ga_k(t) \,,
\end{equation}
for $f\in\Lt(\R)$. Moreover, its operator norm can be bounded by
\begin{align}
\abs{\inner{S_{g,\ga} f}{h}} &\leq \Big ( \sup_{k\in\Z}
(1+b_k^{-1})\norm{\ga_k}_{W(L^{\infty},\ell^1)} \Big )^{1/2} \Big (\esssup_{t\in\R} \sum_{k\in\Z}
\abs{g_k(t)} \Big )^{1/2} \nonumber \\
& \cdot \Big ( \sup_{k\in\Z} (1+b_k^{-1})\norm{g_k}_{W(L^{\infty},\ell^1)} \Big )^{1/2} 
\Big (\esssup_{t\in\R} \sum_{k\in\Z} \abs{\ga_k(t)} \Big )^{1/2} 
\norm{f}_2 \norm{h}_2 \label{eq:bound_general}
\end{align}
for all $f,h\in\Lt(\R)$.
\end{proposition}

\begin{proof}
First assume that  $f, h \in L^2 (\mathbb{R} ) $ are compactly supported. Since $\inner{f}{M_{lb_k}g_k}
= \widehat{(f\bar{g}_k)}(lb_k)$ we can write $S_{g,\ga}$ as
\begin{equation}\label{eq:S}
S_{g,\ga}f(t) = \sum_{k,l\in\Z} \widehat{(f\bar{g}_k)}(lb_k) M_{lb_k} \ga_k(t) =
\sum_{k\in\Z} m_k(t) \ga_k(t)\,,
\end{equation}
where $m_k(t) = \sum_{l\in\Z} \widehat{(f\bar{g}_k)}(lb_k) e^{2\pi i lb_k
t}$, for every $k\in\Z$. The functions $m_k$ are $b_k^{-1}$ periodic and by the
Poisson formula can be written as
\begin{equation}\label{eq:m_k}
m_k(t) = b_k^{-1} \sum_{l\in\Z} (f \bar{g}_k)(t-lb_k^{-1})\,.
\end{equation}
Therefore, substituting (\ref{eq:m_k}) in (\ref{eq:S}) yields the Walnut
representation. 

We next prove the boundedness~\eqref{eq:bound_general}. In the following chain of inequalities, we will use Cauchy-Schwartz inequality for sums and integrals and, since all summands have absolute value, Fubini's theorem to justify  changing  the order of summation and integral. We thus find
\begin{align}
&\abs{\inner{S_{g,\ga} f}{h}} = \absBig{\innerBig{\sum_{k,l\in\Z} b_k^{-1}
\overline{g_k(\cdot -lb_k^{-1})} \ga_k(\cdot) f(\cdot-lb_k^{-1})}{h}} \nonumber \\
&\leq \sum_{k,l\in\Z}b_k^{-1} \int_{\R}
\abs{g_k(t-lb_k^{-1})} \abs{\ga_k(t)}\abs{f(t-lb_k^{-1})}
\abs{h(t)} \, dt \nonumber \\
& \leq \sum_{k,l\in\Z} b_k^{-1} \left [
\int_{\R} \abs{g_k(t-lb_k^{-1})} \abs{\ga_k(t)}\abs{f(t-lb_k^{-1})}^2
\, dt \right ]^{1/2} \left [ \int_{\R} \abs{g_k(t-lb_k^{-1})}
\abs{\ga_k(t)}\abs{h(t)}^2\, dt \right]^{1/2} \nonumber \\
& \leq \left [ \sum_{k,l\in\Z} b_k^{-1} \int_{\R}
\abs{g_k(t)} \abs{\ga_k(t+lb_k^{-1})}\abs{f(t)}^2\, dt
\right]^{1/2} \left [ \sum_{k,l\in\Z} b_k^{-1} \int_{\R}
\abs{\ga_k(t)} \abs{g_k(t-lb_k^{-1})}\abs{h(t)}^2 \, dt
\right]^{1/2} \nonumber \\\label{eq:frame_operator}
& = \left [\int_{\R} \abs{f(t)}^2 \sum_{k,l\in\Z} b_k^{-1}\abs{g_k(t)}
\abs{\ga_k(t-lb_k^{-1})} \, dt \right ]^{1/2} \left [ \int_{\R} \abs{h(t)}^2
\sum_{k,l\in\Z} b_k^{-1}\abs{\ga_k(t)} \abs{g_k(t-lb_k^{-1})} \, dt \right
]^{1/2}\,.
\end{align}
The first term in the last expression can be bounded as follows
\begin{align}
\int_{\R} \abs{f(t)}^2 &\sum_{k,l\in\Z} b_k^{-1}\abs{g_k(t)}
\abs{\ga_k(t-lb_k^{-1})} \, dt =
\sum_{k\in\Z} b_k^{-1}  \int_{\R} \sum_{l\in\Z}
\abs{\ga_k(t-lb_k^{-1})} \abs{f(t)}^2 \abs{g_k(t)}\, dt\nonumber \\
& \leq \sum_{k\in\Z}\left ( b_k^{-1} \esssup_{t\in\R} \sum_{l\in\Z}
\abs{\ga_k(t-lb_k^{-1})}\right ) \int_{\R} \abs{f(t)}^2 \abs{g_k(t)}\, dt
\nonumber \\
& \leq \sup_{k\in\Z}\left (  b_k^{-1} \esssup_{t\in\R}
\sum_{l\in\Z} \abs{\ga_k(t-lb_k^{-1})} \right ) 
\sum_{k\in\Z} \int_{\R} \abs{f(t)}^2 \abs{g_k(t)}\, dt\nonumber\\
\label{eq:frame_operator_cont}
& \leq \sup_{k\in\Z}\left (  b_k^{-1} \esssup_{t\in\R}
\sum_{l\in\Z} \abs{\ga_k(t-lb_k^{-1})} \right ) \left ( \esssup_{t\in\R}
\sum_{k\in\Z} \abs{g_k(t)} \right ) \norm{f}_2^2\,.
\end{align}
Using relation \eqref{eq:wiener_norm} for 
$\esssup_{t\in\R}
\sum_{l\in\Z} \abs{\ga_k(t-lb_k^{-1})}$
 and substituting \eqref{eq:frame_operator_cont} into (\ref{eq:frame_operator}) yields \eqref{eq:bound_general}. The estimate for the second term in  \eqref{eq:frame_operator} is obtained analogously.  By the density of compactly supported functions in
$\Lt(\R)$, the estimate holds for all of $\Lt(\R)$.
\end{proof}

\begin{remark}\label{rem:bound}
From \eqref{eq:frame_operator} in the proof of Proposition~\ref{prop:non_walnut}, it follows that the operator
$S_{g,\ga}$ is also bounded by
\begin{align*}
\abs{\inner{S_{g,\ga} f}{h}} &\leq \left ( \esssup_{t\in\R} \sum_{k,l\in\Z}
b_k^{-1} \abs{g_k(t-lb_k^{-1})}\abs{\ga_k(t)} \right )^{1/2} \\ \nonumber 
& \cdot \left ( \esssup_{t\in\R} \sum_{k,l\in\Z}
b_k^{-1} \abs{\ga_k(t-lb_k^{-1})}\abs{g_k(t)} \right )^{1/2} \norm{f}_2 \norm{h}_2\,.
\end{align*}
\end{remark}

\begin{remark}
In the case of a frame operator $S_{g,g}$, the Walnut representation \eqref{eq:walnut_general} becomes
\begin{equation}\label{eq:walnut}
S_{g,g}f(t) = \sum_{k,l\in\Z} b_k^{-1} \overline{g_k(t-lb_k^{-1})} g_k(t) f(t-lb_k^{-1}) \quad \mbox{a.e.}\,,
\end{equation}
and the above bounds reduce to 
\begin{align*}
\abs{\inner{S_{g,g}f}{h}} &\leq \Big ( \sup_{k\in\Z}
(1+b_k^{-1})\norm{g_k}_{W(L^{\infty},\ell^1)} \Big ) 
\Big (\esssup_{t\in\R} \sum_{k\in\Z} \abs{g_k(t)} \Big ) 
\norm{f}_2 \norm{h}_2 \\ 
\abs{\inner{S_{g,g}f}{h}} &\leq  \left ( \esssup_{t\in\R} \sum_{k,l\in\Z}
b_k^{-1} \abs{g_k(t-lb_k^{-1})}\abs{g_k(t)} \right ) \norm{f}_2 \norm{h}_2
\end{align*}
\end{remark}

\begin{remark}
Note that in the painless case of Theorem~\ref{BDF}, the frame operator reduces to the multiplication operator $S_{g,g} f = \sum_{k\in\Z}G^{g,g}_{k,0} \cdot f = G_0 \cdot f$.
\end{remark}

\begin{remark}
In the standard setting of Gabor frames, i.e. $g_k(t) = g(t-ak)$ for
fixed $a>0$, and $b_k=b$ for all $k\in\Z$, the above bound reduces to the
well know bound
\begin{equation*}
\abs{\inner{S_{g,g}f}{h}} \leq (1+a^{-1})(1+b^{-1}) \norm{g}_{W(L^{\infty},\ell^1)}^2
\norm{f}_2 \norm{h}_2
\end{equation*}
\end{remark}

\subsection{Existence of nonstationary Gabor frames}\label{Se:Exis}

For  windows $g_k$ that are neither compactly supported nor bandlimited,  we
are interested in the existence of  frames of the form
$\mathcal{G}(\mathbf{g},\mathbf{b})$  and in the construction of the involved
parameters. The
following theorem derives a sufficient condition for the existence of
nonstationary Gabor frames and shows that this condition can be satisfied.

In this and the subsequent sections, $[b_L,b_U]$, $[p_L,p_U]$, $[C_L,C_U]$ are compact intervals  of positive real numbers. 

\begin{theorem}\label{thm:nonstationary_frames_1}
Let $\mathbf{g} = \{g_k \in W(L^{\infty},\ell^1):\; k\in\Z\}$ be a set of windows such that 
\begin{itemize}
\item[i)] for some positive constants $A_0, B_0$
\begin{equation}\label{Eq:uplowbound}
0<A_0\leq\sum_{k\in\Z} \abs{g_k(t)}^2 \leq B_0 <\infty\,\mbox{ a.e. };
\end{equation}
\item[ii)] for all  $k\in\Z$, the windows decay polynomially around a $\delta$-separated set $\{ a_k\, : \, k\in \Z \}$ of time-sampling points $a_k$
\begin{equation}\label{eq:g_k}
\abs{g_k(t)} \leq C_k(1+\abs{t-a_k})^{-p_k} \,, 
\end{equation}
where  $p_k\in [p_L,p_U]\subset\mathbb{R}$, $p_L>2$  and $C_k \in [C_L, C_U]$.
\end{itemize}
Then there exists a sequence $\{b_k^0\}_{k\in\Z}$, such that for $b_k \leq
b_k^0$, $k\in\Z$, the nonstationary Gabor
system $\mathcal{G}(\mathbf{g},\mathbf{b})$ forms a frame
for $\Lt(\R)$.
\end{theorem}

\begin{proof}
Let $f\in\Lt(\R)$. Applying \eqref{eq:walnut}, we write
\begin{equation*}
\inner{S_{g,g}f}{f} = \int_{\R} \sum_{k\in\Z} b_k^{-1} \abs{g_k(t)}^2 \abs{f(t)}^2 \,
dt + \int_{\R} \sum_{l\in\Z\setminus \{0\}} \sum_{k\in\Z}
b_k^{-1}g_k(t)\overline{g_k(t-lb_k^{-1})} f(t-lb_k^{-1})\overline{f(t)}\,dt
\end{equation*}
Using similar arguments as in the derivation of (\ref{eq:frame_operator}), we obtain
\begin{align*}
\absBig{\int_{\R} \sum_{l\in\Z\setminus \{0\}} \sum_{k\in\Z}
&b_k^{-1}g_k(t)\overline{g_k(t-lb_k^{-1})} f(t-lb_k^{-1})\overline{f(t)}\,dt}
\leq \nonumber \\
&\leq \left (\esssup_{t\in\R} \sum_{l\in\Z\setminus \{0\}} \sum_{k\in\Z}
b_k^{-1}\abs{g_k(t)}\abs{g_k(t-lb_k^{-1})} \right ) \, \norm{f}_2^2
\nonumber \\
&\leq \max_{k\in\Z}\{b_k^{-1}\} \, \underbrace{ 
\sum_{l\in\Z\setminus \{0\}} \left ( \esssup_{t\in\R}
\sum_{k\in\Z} \abs{g_k(t)}\abs{g_k(t-lb_k^{-1})} \right )}_{R}
\, \norm{f}_2^2 \,.
\end{align*}
Therefore, lower and upper  frame bounds are obtained from
\begin{align}\label{eq:frame_bounds}
\inner{S_{g,g} f}{f} \, \norm{f}_2^{-2} \geq & \min_{k\in\Z}\{b_k^{-1}\}\,\left
( \essinf_{t\in\R} \sum_{k\in\Z} \abs{g_k(t)}^2  - \,
\frac{\max_{k\in\Z}\{b_k^{-1}\}}{\min_{k\in\Z}\{b_k^{-1}\}} R
\right)\\
\inner{S_{g,g} f}{f} \, \norm{f}_2^{-2} \leq & \max_{k\in\Z}\{b_k^{-1}\} \, \left (
\esssup_{t\in\R} \sum_{k\in\Z}\abs{g_k(t)}^2 + \, R \right)\,,\nonumber
\end{align}
We need to construct a sequence of $b_k$, $k\in\Z$, such that for all
$f\in\Lt(\R)$, \eqref{eq:frame_bounds} is bounded away from zero.

Let $\epsilon <C_L$ and consider the sequence of frequency shifts  $b_k = 
(\frac{ \epsilon}{C_k})^{1/p_k}$. Then 
$\min_{k\in\Z}\{b_k^{-1}\} \geq  (C_L\epsilon^{-1})^{1/p_2}$, $\max_{k\in\Z}\{b_k^{-1}\} \leq  (C_U
\epsilon^{-1})^{1/p_1}$ and
\begin{equation}\label{eq:max/min}
\frac{\max_{k\in\Z}\{b_k^{-1}\}}{\min_{k\in\Z}\{b_k^{-1}\}} \leq C_U^{1/p_L}
C_L^{-1/p_U} \, \,\epsilon^{1/p_U - 1/p_L}\,.
\end{equation}
Since $(1+\abs{x+y})^{-p} \leq (1+\abs{x})^p(1+\abs{y})^{-p}$ for
$x,y\in\R$ and $p\geq 0$, using \eqref{eq:g_k}, we have, for some $\mu$ with $p_L-2> \mu>0$:
\begin{align*}
\abs{g_k(t)}\abs{g_k(t-lb_k^{-1})} &\leq C_k^2 (1 + \abs{t-a_k})^{-p_k}
(1+\abs{t-a_k-lb_k^{-1}})^{-p_k+(1+\mu)} \nonumber \\
&\leq C_k^2 (1 + \abs{t-a_k})^{(1+\mu)} (1 + \abs{l}b_k^{-1})^{-p_k+(1+\mu)}\nonumber \\
&\leq C_k^2 (1 + \abs{t-a_k})^{-(1+\mu)} \, \, \abs{l}^{-p_k+(1+\mu)}
\,\,b_k^{p_k-(1+\mu)}\nonumber \\
&= \underbrace{C_k^{1+(1+\mu)/p_k}}_{E_k} (1 + \abs{t-a_k})^{-(1+\mu)} \, \,
\abs{l}^{-p_k+(1+\mu)} \,\, \epsilon^{1-(1+\mu)/p_k} \,. 
\end{align*}
Hence,	
\begin{align}\label{eq:sup_g}
\sum_{k\in\Z} \abs{g_k(t)}&\abs{g_k(t-lb_k^{-1})} \leq
\sum_{k\in\Z} E_k (1 + \abs{t-a_k})^{-(1+\mu)} \, \,
\abs{l}^{-p_k+(1+\mu)} \,\, \epsilon^{1-(1+\mu)/p_k}  \nonumber \\
&\leq \max_{k\in\Z} E_k \,
\, \abs{l}^{-p_L+(1+\mu)}\,\, \epsilon^{1-(1+\mu)/p_L}\, \sum_{k\in\Z} (1 +
\abs{t-a_k})^{-(1+\mu)} \nonumber \\
&\leq \max_{k\in\Z} E_k \, \, \abs{l}^{-p_L+(1+\mu)}\,\, \epsilon^{1-(1+\mu)/p_L}\,
2 \left ( 1 + (1+\delta)^{-(1+\mu)} (\delta^{-1} + 1+\mu)\mu^{-1} \right
)\,,
\end{align}
where the last estimate follows from Lemma~\ref{lem:estimates}~(b). Summing the expression
(\ref{eq:sup_g}) over $l \in\Z\setminus\{0\}$ using
Lemma~\ref{lem:estimates}~<	, we  see that $R$, as a function of $\epsilon$,
behaves like $\epsilon^{1-(1+\mu)/p_L}$, i.e. $R \approx \epsilon^{1-(1+\mu)/p_L}$, and $R$ tends to $0$ for $\epsilon\rightarrow
0$. Moreover, 
\begin{equation*}
\frac{\max_{k\in\Z}\{b_k^{-1}\}}{\min_{k\in\Z}\{b_k^{-1}\}} R \approx
\epsilon^{1 - (2+\mu)/p_L + 1/p_U} 
\end{equation*}
can be made arbitrarily small by choosing $\epsilon$ small since $1 -
(2+\mu)/p_L + 1/p_U >0$. Therefore, if
$\epsilon_0$ is such that for $b_k^0 := (\frac{\epsilon_0}{C_k})^{1/p_k}$, $\frac{\max_{k\in\Z}\{b_k^{-1}\}}{\min_{k\in\Z}\{b_k^{-1}\}} R < A_0$, then $\{ M_{lb_k} g_k \}_{k,l\in\Z}$ is a frame for all $b_k \leq b_k^0$.
\end{proof}

For completeness, we state the equivalent result for analysis windows $g_k$ with polynomial decay in the frequency domain.
\begin{corollary}\label{cor:nonstationary_frames_FS}
Let $\mathbf{g} = \{g_k \in L^2 (\mathbb{R} ):\, \hat{g}_k\in W(L^{\infty},\ell^1),\, k\in\Z\}$ be a set of  windows such
that 
\begin{itemize}
\item[i)] for some positive constants $A_0, B_0$
\begin{equation}\label{Eq:uplowbound1}
0<A_0\leq\sum_{k\in\Z} \abs{\hat{g}_k(t)}^2 \leq B_0 <\infty\,\mbox{ a.e. };
\end{equation}
\item[ii)] for all  $k\in\Z$, the windows decay polynomially around a $\delta$-separated set $\{ b_k\, : \, k\in \Z \}$ of frequency-sampling points $b_k$:
\begin{equation}\label{eq:g_k_1}
\abs{\hat{g}_k(t)} \leq C_k(1+\abs{t-b_k})^{-p_k} \,, 
\end{equation}
where  $p_k$ and $C_k$ are chosen as in Theorem~\ref{thm:nonstationary_frames_1}.
\end{itemize}
Then there exists a sequence $\{a_k^0\}_{k\in\Z}$, such that for $a_k \leq
a_k^0$, $k\in\Z$, the nonstationary Gabor system  $\{T_{la_k} g_k\, : \, k,l\in\Z \}$ forms a frame
for $\Lt(\R)$.
\end{corollary}

\subsubsection{Nonstationary Gabor frames on modulation spaces}
Modulation spaces, cf.~\cite{fe81-3,gr01}, are considered as the appropriate function spaces for time-frequency analysis and in particular, for the study of Gabor frames. By their definition, modulation spaces require decay in both time and frequency.
Under additional assumptions on the windows $g_k$, the collection 
$\mathcal{G}(\mathbf{g},\mathbf{b})$ is a frame for all modulation spaces $M^p$,
$1\leq p \leq
\infty$.

\begin{proposition}\label{thm:modulation}
Let $\mathcal{G}(\mathbf{g},\mathbf{b})$ be a frame for $\Lt(\R)$ satisfying the
uniform estimate
\begin{equation}\label{eq:stft}
\abs{V_\phi g_k(x,\om)} \leq C (1+\abs{x-a_k})^{-r-2}
(1+\abs{\om})^{-r-2}\,, \quad r>2
\end{equation}
where $\phi$ is a Gaussian window. Then the frame operator $S$ is invertible
simultaneously on all modulation spaces $M^p$ for $1\leq p\leq \infty$.
\end{proposition}
\begin{proof} Notice, that
\begin{equation*}
\abs{V_\phi g_{k,l}(x,\om)}  \leq C(1+\abs{(x-a_k,\om - lb_k)})^{-r-2}\,,
\end{equation*}
since $(1+\abs{x} +\abs{\om} )^{-r} \geq (1+\abs{x})^{-r}(1+\abs{\om})^{-r}$ and the weights 
$(1+\abs{(x,\om)})^r$ and $(1+\abs{x}+\abs{\om})^r$ are equivalent.\\

A result on Gabor
molecules~\cite{G04} states that,  if an $\Lt$-frame $\{
g_z\,:\,z=(z_1,z_2)\in\mathcal{Z}\subseteq \Rt \}$, where $\mathcal{Z}$ is separable, 
satisfies the uniform estimate $\abs{V_{\phi} g_z(x,\om)} \leq C(1+\abs{(x-z_1,\om-z_2)})^{-r-2}$, 
then the frame operator $Sf = \sum_{z\in\mathcal{Z}}\inner{f}{g_z} g_z$ is invertible simultaneously on all
$M^p$ for each $1\leq p \leq \infty$. The result hence follows from 
condition \eqref{eq:stft}. \end{proof}

\subsection{Constructing nonstationary Gabor frames}\label{Se:apnGF}
Theorem~\ref{thm:nonstationary_frames_1} shows that for windows with sufficient uniform decay, nonstationary Gabor frames can always be constructed by choosing sufficient density in the frequency samples. 
In the present section we assume the existence of a certain  nonstationary Gabor frame and explicitly construct a  new  frame  by exploiting  the prior information about  the original one.  This is a situation of practical relevance, since we may often be interested in using windows that decay fast and are negligible outside a support of interest. In particular, we will use the fact that painless nonstationary Gabor frames are easily constructed and deduce new frames from painless frames. The new frames thus obtained will be called \emph{almost painless nonstationary Gabor frames}.

We will subsequently assume    $b_k\in[b_L,b_U]$ for frequency-shift parameters and  we let $C_k\in [C_L,C_U]$ and  $p_k\in[p_L,p_U]$ with $p_L>1$ for the constants involved in the decay assumptions for the analysis windows. We then work with the following constants that depend on the separation of the sampling points, the decay of the windows and the frequency-sampling parameters: 
\begin{equation}\label{Def:Const}
E_1 = 1 + \frac{\delta^{-1} + p_L}{(1+\delta)^{p_L} (p_L-1)}\,\mbox{ and }\,
E_2 = 1 + \frac{b_U+p_U}{(1+b_U^{-1})^{p_L}(p_L-1)}
\end{equation}

We first consider nonstationary Gabor frames obtained by perturbation of a known frame. This result is in the spirit of similar results for regular Gabor frames \cite{CF03,chhe97}.

\begin{proposition}\label{thm:nonstationary_frames}
Let $\{ a_k \, : \, k\in \Z\}$ be a $\delta$-separated set and
$\mathcal{G}(\mathbf{h},\mathbf{b})$ a nonstationary Gabor frame
with frame bounds $A_h$, $B_h$ and  frame operator $S_{h,h}$. Let $g_k\in\Lt(\R)$
be a set of windows such that for all $k\in\Z$ and for almost all $t \in \R$
\begin{equation}\label{eq:tails}
\abs{g_k(t) - h_k(t)} \leq C_k (1+\abs{t-a_k})^{-p_k}.
\end{equation}
If $C_U < \sqrt{A_h\lambda^{-1}}$ for
\begin{equation}\label{eq:lambda}
\lambda = 4b_L^{-1}\cdot E_1 \cdot E_2 , 
\end{equation}
then $\mathcal{G}(\mathbf{g},\mathbf{b})$ is a frame for $\Lt(\R)$ with frame
bounds $A=A_h(1-\sqrt{C_U^2\lambda A_h^{-1}})^2$ and
$B=B_h(1+\sqrt{C_U^2\lambda B_h^{-1}})^2$.
\end{proposition}

\begin{proof}
By applying Cauchy-Schwartz inequality, it is easy to see that,   for $\mathcal{G}(\mathbf{g},\mathbf{b})$ to be a frame for
$\Lt(\R)$, it suffices to show that $\sum_{k,l\in\Z} \abs{\inner{f}{g_{k,l} -
h_{k,l}}}^2 \leq R \norm{f}_2^2$ for some $R < A_h$, also cf.~\cite[Proposition~4.1.]{chhe97}. Then, frame bounds of $\mathcal{G}(\mathbf{g},\mathbf{b})$ can be taken as $A_h (1-\sqrt{R/A_h})^2$ and $B_h(1+\sqrt{R/B_h})^2$.

To obtain the required error bound, we let   $\psi_k(t) := g_k - h_k$ and use the estimate given in \eqref{eq:frame_operator_cont}:
\begin{equation}\label{eq:perturbation}
\sum_{k,l\in\Z} \abs{\inner{f}{\psi_{k,l}}}^2 \leq \sup_{k\in\Z} \Big (
b_k^{-1} \esssup_{t\in\R} \sum_{l\in\Z} \abs{\psi_k (t-lb_k^{-1})} \Big ) \Big (\mbox{ess
sup}_{t\in\R} \sum_{k\in\Z} \abs{\psi_k(t)} \Big )
\norm{f}_2^2\,.
\end{equation}
The first  term of the above in the above inequality is bounded  by assumption \eqref{eq:tails} and Lemma~\ref{lem:estimates}~(b):
\begin{equation*}
\sum_{k\in\Z} \abs{\psi_k(t)} \leq C_U \sum_{k\in\Z} (1 +
\abs{t-a_k})^{-p_k} \leq C_U\sum_{k\in\Z} (1 +
\abs{t-a_k})^{-p_1} \nonumber \leq 2C_U E_1 \,.
\end{equation*}
To bound the second term, note that $\sum_{l\in\Z} (1+\abs{t-a_k - lb_k^{-1}})^{-p_k}$ is
$b_k^{-1}-$periodic, therefore we can simplify
\begin{equation*}
\esssup_{t\in\R}\sum_{l\in\Z} \abs{\psi_k(t-lb_k^{-1})} \leq C_k
\mbox{ess sup}_{t \in [0,b_k^{-1}]} \sum_{l\in\Z} (1+\abs{t -
lb_k^{-1}})^{-p_k}\,.
\end{equation*}
Hence,  by Lemma~\ref{lem:estimates}, we
obtain for $t\in [0,b_k^{-1}]$:
\begin{align*}
\sum_{l\in\Z} (1+\abs{t - lb_k^{-1}})^{-p_k} &\leq 1 + \sum_{l=1}^{\infty}
(1+\abs{t - lb_k^{-1}})^{-p_k} + \sum_{l=1}^{\infty} (1+t +
lb_k^{-1})^{-p_k} \nonumber \\
(1+(l-1)b_k^{-1})^{-p_k} \nonumber \\
& = 2\sum_{l=0}^{\infty} (1+lb_k^{-1})^{-p_k} = 2 \Big (1+ \sum_{l=1}^{\infty}
(1+lb_k^{-1})^{-p_k}\Big ) \nonumber  \leq 2 E_2.
\end{align*}
Gathering all the estimates, we obtain
\begin{align*}
\sum_{k,l\in\Z} \abs{\inner{f}{\psi_{k,l}}}^2 &\leq C_U^2
4b_1^{-1}E_1\cdot E_2\norm{f}_2^2
 = C_U^2 \lambda \norm{f}_2^2\,.
\end{align*}
By assumption $C_U^2 \lambda<A_h$, and the proof is complete.
\end{proof}

Using Proposition~\ref{thm:nonstationary_frames} we next construct a special class of nonstationary Gabor frames by relying on knowledge of a painless nonstationary Gabor frame. We construct new windows which are no more compactly supported, but coincide with the known, compact windows on their support. We call the resulting new systems 
\emph{almost painless} nonstationary Gabor frames. 

\begin{corollary}\label{cor:nonstationary_frames_2}
Let $\mathbf{g}=\{g_k\in W(L^{\infty},\ell^1): \;k\in\Z\}$ be a set of
windows, and let $I_k$ be the intervals $I_k =[a_k-(2b_k)^{-1}, a_k+(2b_k)^{-1}]$ where $\{ a_k\, : \, k\in\Z \}$ forms a $\delta$-separated set. Assume that $\mathcal{G}(\mathbf{h},\mathbf{b})$, where $h_k = g_k\chi_{I_k}$, is a Gabor frame with lower frame bound $A_h$, and that for $\psi_k = g_k - h_k$,  all $k\in\Z$ and  almost all $t in\R$
\begin{equation}\label{Eq:Decayass}
\abs{\psi_k(t)} \leq \left \{ \begin{array}{ll} C_k
(1+t-a_k-(2b_k)^{-1})^{-p_k}\,,
& t > a_k+(2b_k)^{-1}\,; \\
0\,, & t\in I_k\,; \\
C_k (1-t+a_k-(2b_k)^{-1})^{-p_k}\,, & t < a_k-(2b_k)^{-1}\,. \end{array} \right. 
\end{equation}
 If $C_U  < \sqrt{A_h \lambda^{-1}}$
for
$\lambda = 4b_L^{-2}\cdot \delta^{-1} \cdot E_1 \cdot E_2$, 
then $\mathcal{G}(\mathbf{g},\mathbf{b})$ forms a nonstationary Gabor frame for
$\Lt(\R)$.
\end{corollary}

\begin{proof}
The proof follows the steps of the proof of Proposition~\ref{thm:nonstationary_frames} with small changes 
on how to approximate the terms in \eqref{eq:perturbation}.
First observe, that for any $t\in\mathbb{R}$ and all $k$, we have
\begin{equation*}
\abs{\psi_k(t)} \leq C_k \left [ (1+\abs{t-a_k-(2b_k)^{-1}})^{-p_k} +
(1+\abs{t-a_k+(2b_k)^{-1}})^{-p_k}\right ]
\end{equation*}
Since the frequency shifts $b_k$ are taken from a finite interval and the set $\{a_k\, : \, k\in\Z \}$ is $\delta-$separated, the sets $\Gamma^+ = \{a_k+ (2b_k)^{-1}\, : \, k\in\Z \}$ and $\Gamma^- = \{a_k- (2b_k)^{-1}\, : \, k\in\Z \}$ are relatively $\delta-$separated with $\mbox{rel}(\Gamma) =\mbox{rel}(\Gamma^+) = \mbox{rel}(\Gamma^-) = \lfloor (2b_L \delta)^{-1} \rfloor$. Therefore, by Remark~\ref{rem:estimate}, it follows that
\begin{align*}
\sum_{k\in\Z} \abs{\psi_k(t)} &\leq C_U \sum_{k\in\Z}
 (1+\abs{t-a_k-(2b_k)^{-1}})^{-p_L} + C_U \sum_{k\in\Z}
(1+\abs{t-a_k+(2b_k)^{-1}})^{-p_L} \nonumber \\
&\leq 4\, C_U \, \mbox{rel}(\Gamma) \, (1+(1+\delta)^{-p_L}(\delta^{-1}+p_L)(p_L-1)^{-1})\,.
\end{align*}

Now,  the expression $\sum_{l\in\Z} \abs{\psi_k(t-lb_k^{-1})}$ is
$b_k^{-1}-$periodic. Let $t\in I_k$, then, by \eqref{Eq:Decayass}
\begin{align*}
\sum_{l\in\Z} \abs{\psi_k(t-lb_k^{-1})} &\leq C_k \left [ \sum_{l > 0}
(1+a_k-(2b_k)^{-1} -t +lb_k^{-1})^{-p_k} \right. \nonumber \\ 
&\left. 
+ \sum_{l < 0} (1-a_k-(2b_k)^{-1} + t -lb_k^{-1})^{-p_k}\right ] \nonumber \\
&\leq C_k 2\sum_{l=0}^{\infty} (1+lb_k^{-1})^{-p_k} \leq 2C_k
(1+(1+b_k^{-1})^{-p_k}(b_k+p_k)(p_k-1)^{-1})\,,
\end{align*}
where the last estimate follows from Lemma~\ref{lem:estimates}(a). 
\end{proof}



\section{Examples}\label{Se:Ex}

We illustrate our theory with two examples. In both examples,  we  consider a basic window and dilations by $2$ and $\frac{1}{2}$, respectively. Since the dilation parameters take only three different values,
there are three kinds of windows,  with support size $1/2$, $1$ and $2$, respectively. Note that, while theoretically possible, sudden changes in the shape and width of adjacent windows turn out to be undesirable for applications, hence we only allow for stepwise change in dilation parameters.

In the first example we  consider a frame that arises as  perturbation of a painless
nonstationary Gabor frame. The perturbation consists in the application of a bandpass filter in order to obtain windows with compact support in the frequency domain. 

\begin{example}\label{Ex1}
Let $h$ be a Hann or raised cosine window, i.e. $h(t)=0.5 + 0.5\cos(2\pi t)$ for
$t\in[-1/2,1/2]$, and zero otherwise. We construct a painless nonstationary Gabor frame by dilating $h$ by $2 $ or $\frac{1}{2}$, respectively: Let $\{s_k\}_{k\in\Z}$ be a sequence with values from the set $\{-1,0,1\}$ with the restriction that   $\abs{s_k - s_{k-1}}\in \{0,1\}$ to avoid sudden changes between adjacent windows.  We then define corresponding shift-parameters by setting $a_0 = 0$ and 
\begin{align*}
a_{k+1}&= a_k + 2^{-s_k}\cdot \frac{5}{6} \quad \quad \mbox{if} \quad s_k >
s_{k+1}\,,\\
a_{k+1}&= a_k + 2^{-s_k+1}\cdot \frac{1}{3} \quad \mbox{if} \quad
s_k=s_{k+1}\,,\\
a_{k+1} &= a_k + 2^{-s_{k+1}}\cdot \frac{5}{6} \quad \mbox{if} \quad s_k <
s_{k+1}\,.
\end{align*}
The points $a_k$, $k\in\mathbb{Z}$, form a separated set with minimum separation $\delta = 1/3$.
Setting $b_k = 2^{s_k}$  and $h_k(t)= T_{a_k}\sqrt{2^{s_k}}h(2^{s_k}t)$, the system $\{M_{lb_k}h_k\, : \, k,l\in\Z\}$
forms a painless nonstationary Gabor frame with lower frame bound
$A_h=0.5$.

Let $\Om=0.02$ and $\phi$ be a bandlimited filter given by
\begin{equation*}
\widehat{\phi}(\om) = 0.5 + 0.5\cos(2\pi\Om^{-1}\om) \mbox{ on its support } [-\Om/2,\Om/2]\,.
\end{equation*}
We build new windows $g_k$ by convolving $\phi$ with $h_k$. The windows
$g_k :=
\conv{\phi}{h_k}$ are no more compactly supported.
Since $\abs{\phi(t)} \leq \Om (1+\abs{t})^{-3}$,  we rely on \cite[Theorem~9.9]{ru66} to deduce the following bound, with 	$C' = \norm{h_k}_{\infty} \frac{\Om}{2}$, $I_k = [a_k-2^{-s_k'},a_k+2^{-s_k'}]$ and $s_k'= s_k+1$:
\begin{equation*}
\abs{g_k(t) - h_k(t)} \leq  C'\left
\{ \begin{array}{ll} (1+(t-a_k)-2^{-s_k'})^{-2} - (1+(t-a_k)+2^{-s_k'})^{-2} & t>a_k+
2^{-s_k'}\\
2 - (1+(t-a_k)+2^{-s_k'})^{-2} - (1-(t-a_k)+2^{-s_k'})^{-2} &  t\in I_k\\
(1-(t-a_k)-2^{-s_k'})^{-2} - (1-(t-a_k)+2^{-s_k'})^{-2} & t<a_k -2^{-s_k'}
\end{array}\right.
\end{equation*}
We obtain the joint  bound, 
$
\abs{g_k(t) - h_k(t)} \leq C_U(1+\abs{t-a_k})^{-2}$  by setting  $ C_k 
= C'\cdot (1+2^{-s_k'})^{2}$ for all $k\in\Z $
and  $C_U =
\max_{k\in\Z}C_k = 0.0282  < \sqrt{A_h\lambda^{-1}}=0.0768$, with $\lambda$ as
defined in \eqref{eq:lambda}. Thus, by Proposition~\ref{thm:nonstationary_frames},   $\{ M_{lb_k} g_k\, : \, k,l\in\Z \}$ is a
Gabor frame with a lower frame bound $A=0.2$. 
\end{example}
\begin{remark}
Note that the construction presented in the Example~\ref{Ex1} is of particular interest for constructing   frequency-adaptive frames with windows that are compactly supported in time. This is a situation of considerable interest in applications, since it allows for real-time implementation with finite impulse response filters,~cp.~\cite{evdoma12}.
\end{remark}

In the second example we
construct a nonstationary Gabor frame by applying 
Corollary~\ref{cor:nonstationary_frames_2}. The windows of the new frame coincide with the windows of a painless frame on their support.  The windows in this example are constructed in analogy to the windows used in \emph{scale frames},  introduced in~\cite{badohojave11} to automatically improve the resolution of transients in audio signals. 

\begin{example}
As in the previous example, let $s_k \in \{-1,0,1\}$  with $\abs{s_k-s_{k-1}} \in \{0,1\}$ for all $k\in\Z$. We consider a sequence of windows
$g_k$ that are translated and dilated versions of the Gaussian window
$g(t) = e^{-\pi(2.5t)^2}$: $g_k(t) = T_{a_k}\sqrt{2^{s_k}}g(2^{s_k}t)$ 
with $a_0=0$ and
\begin{align*}
a_{k+1} &= a_k + 2^{-s_{k+1}-1} \,\quad \,\,\mbox{if} \quad s_k=s_{k+1}\,,\\
a_{k+1} &= a_k + \frac{1}{3}\cdot 2^{-s_{k+1}} \quad \mbox{if} \quad
s_k>s_{k+1}\,,\\
a_{k+1} &= a_k + \frac{1}{3}\cdot 2^{-s_k} \,\,\, \quad \,\mbox{if} \quad
\,\, s_k<s_{k+1}\,.
\end{align*}
Here, the $\{a_k\, : \, k\in\Z\}$ are separated
with minimum distance  $\delta = 1/4$.
We arrange the windows as follows: after each change of window size, no change is allowed in the next step; in other words, each window has at least one neighbor of the same size.  An example of the arrangement is shown in Figure~\ref{fig.arrangement}.

\begin{figure}[h]
\centerline{\includegraphics[width=\textwidth]{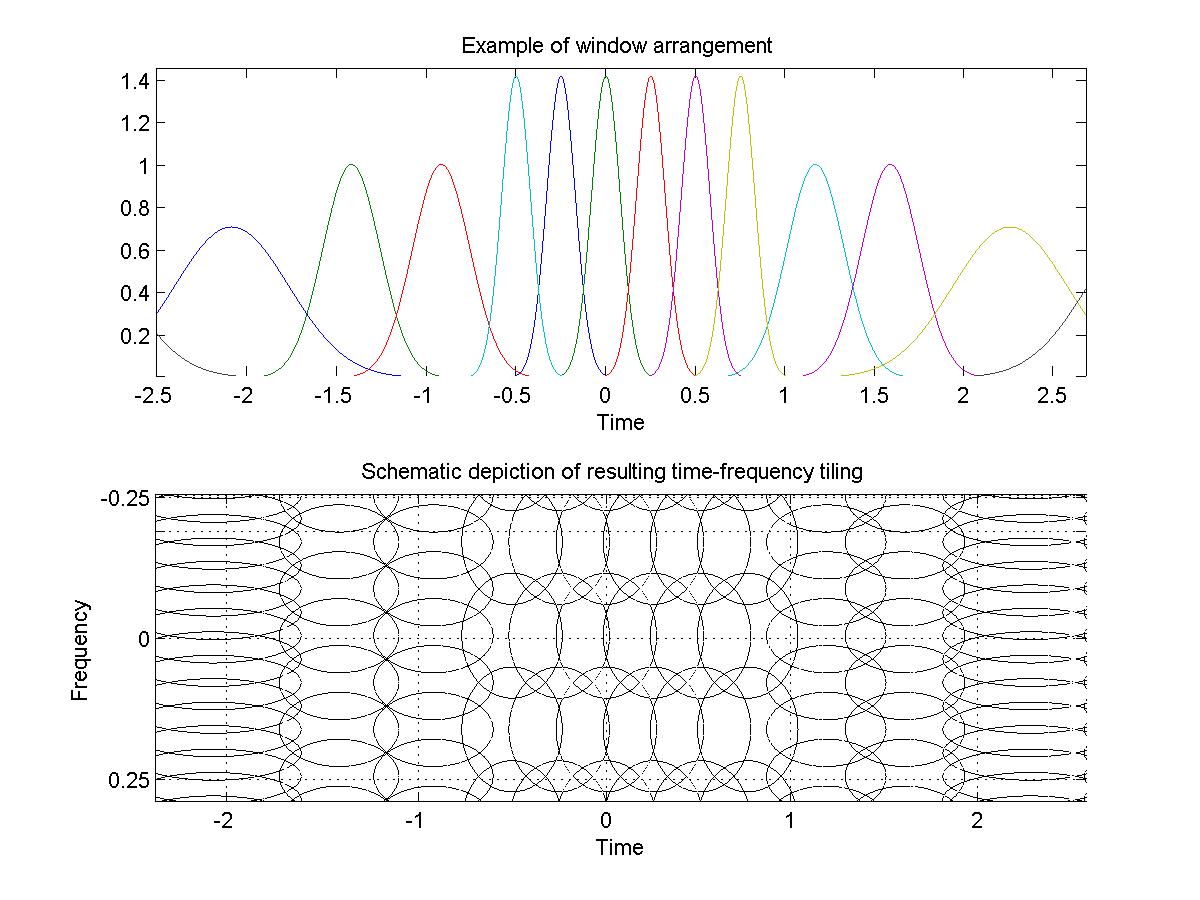}}
 	\caption{An example for the arrangement of dilated windows in Example 2}
	\label{fig.arrangement}
\end{figure}

Let $I_k=[a_k-2^{-s_k-1},a_k + 2^{-s_k-1}]$ and define a new set of windows
by $h_k(t) = g_k(t)\chi_{I_k}$. 
Then $\{M_{lb_k} h_k\, : \, k,l\in\Z\}$ with $b_k = 2^{s_k}$ is a painless nonstationary Gabor frame. 
By numerical calculations, its lower frame bound is $A_h = 0.1609$ and  $\psi_k(t) = g_k(t)-h_k(t)$
can be bounded by
\begin{equation*}
\abs{\psi_k(t)} \leq  \left \{ \begin{array}{ll}
\sqrt{2^{s_k}}g(1/2)(1+t-a_k-2^{-s_k-1})^{-19} & t > a_k+2^{-s_k-1}\\
    0 & t\in I_k \\
\sqrt{2^{s_k}}g(1/2)(1+a_k-2^{-s_k-1}-t)^{-19} & t < a_k-2^{-s_k-1}\,.
\end{array} \right.
\end{equation*}
From Proposition~{\ref{cor:nonstationary_frames_2}} it follows that
$4(\delta b_L^2)^{-1}\cdot C_U^2\cdot E_1 \cdot E_2 = 0.0071<	A_h$, and $\{ M_{lb_k} T_{a_k}g_k \, : \, k,l\in\Z\}$ is a
nonstationary Gabor frame with a lower frame bound  $A=0.1538$.
\end{example}

%

\section{Acknowledgement}
This work was supported by the WWTF project {\em Audiominer} (MA09-24) and  the Austrian Science Fund (FWF):[T384-N13] {\em Locatif}.


\begin{thebibliography}{10}

\bibitem{alcamo04-1}
A.~{A}ldroubi, C.~A. {C}abrelli, and U.~{M}olter.
\newblock {W}avelets on {I}rregular {G}rids with {A}rbitrary {D}ilation
  {M}atrices, and {F}rame {A}toms for {L}2({R}d).
\newblock {\em {A}ppl. {C}omput. {H}armon. {A}nal.}, {S}pecial {I}ssue on
  {F}rames {I}{I}.:119--140, 2004.

\bibitem{badohojave11}
P.~{B}alazs, M.~{D}{\"o}rfler, F.~{J}aillet, N.~{H}olighaus, and G.~A.
  {V}elasco.
\newblock {T}heory, implementation and applications of nonstationary {G}abor
  {F}rames.
\newblock {\em {J}. {C}omput. {A}ppl. {M}ath.}, 236:1481--1496, 2011.

\bibitem{CF03}
I.~{C}arrizo and S.~{F}avier.
\newblock {P}erturbation of wavelet and {{G}}abor frames.
\newblock {\em {A}nal. {T}heory {A}ppl.}, 19(3):238--254, 2003.


\bibitem{ch03}
O.~{C}hristensen.
\newblock {\em {A}n {I}ntroduction to {F}rames and {R}iesz {B}ases.}
\newblock {A}pplied and {N}umerical {H}armonic {A}nalysis. {B}irkh{\"a}user,
  2003.

\bibitem{chfazo01}
O.~{C}hristensen, S.~{F}avier, and F.~{Z}{\'o}.
\newblock {I}rregular wavelet frames and {G}abor frames.
\newblock {\em {A}pprox. {T}heory {A}ppl.}, 17(3):90--101, 2001.

\bibitem{chhe97}
O.~{C}hristensen and C.~{H}eil.
\newblock {P}erturbations of {B}anach frames and atomic decompositions.
\newblock {\em {M}ath. {N}achr.}, 185:33--47, {D}ecember 1997.

\bibitem{dagrme86}
I.~{D}aubechies, A.~{G}rossmann, and Y.~{M}eyer.
\newblock {P}ainless nonorthogonal expansions.
\newblock {\em {J}. {M}ath. {P}hys.}, 27(5):1271--1283, {M}ay 1986.


\bibitem{do11}
M.~{D}{\"o}rfler.
\newblock {Q}uilted {G}abor frames - {A} new concept for adaptive
  time-frequency representation.
\newblock {\em {A}dvances in {A}pplied {M}athematics}, 47(4):668 -- 687, {O}ct.
  2011.

\bibitem{dogrhove12}
N.~{H}olighaus, M.~{D}{\"o}rfler, G.~{V}elasco, and T.~{G}rill.
\newblock {A} framework for invertible, real-time constant-{Q} transforms.
\newblock {\em preprint}, submitted, http://www.univie.ac.at/nonstatgab/slicq,
2012. 



\bibitem{dusc52}
R.~J. {D}uffin and A.~C. {S}chaeffer.
\newblock {A} class of nonharmonic {F}ourier series.
\newblock {\em {T}rans. {A}mer. {M}ath. {S}oc.}, 72:341--366, 1952.

\bibitem{evdoma12}
G.~{E}vangelista, M.~{D}{\"o}rfler, and E.~{M}atusiak.
\newblock {P}hase {V}ocoders {W}ith {A}rbitrary {F}requency {B}and {S}election.
\newblock {\em {P}roceedings of the 9th {S}ound and {M}usic {C}omputing
{C}onference, {J}uly 11-14th 2012 {K}openhagen}, 2012. 

\bibitem{faza95}
S.~J. {F}avier and R.~A. {Z}alik.
\newblock {O}n the stability of frames and {R}iesz bases.
\newblock {\em {A}ppl. {C}omput. {H}armon. {A}nal.}, 2(2):160--173, 1995.

\bibitem{fe81-3}
H.~G. {F}eichtinger.
\newblock {O}n a new {S}egal algebra.
\newblock {\em {M}onatsh. {M}ath.}, 92:269--289, 1981.


\bibitem{ga46}
D.~{G}abor.
\newblock {T}heory of communication.
\newblock {\em {J}. {I}{E}{E}}, 93(26):429--457, 1946.

\bibitem{gr01}
K.~{G}r{\"o}chenig.
\newblock {\em {F}oundations of {T}ime-{F}requency {A}nalysis}.
\newblock {A}ppl. {N}umer. {H}armon. {A}nal. {B}irkh{\"a}user {B}oston, 2001.

\bibitem{G04}
K.~{G}r{\"o}chenig.
\newblock {\em {L}ocalization of {F}rames, {B}anach {F}rames, and the
  invertibility of the frame operator}.
\newblock {P}roceedings of {S}{P}{I}{E}, {S}an {D}iego, 2004.

\bibitem{ma92-2}
H.~{M}alvar.
\newblock {\em {S}ignal {P}rocessing with {L}apped {T}ransforms}.
\newblock {B}oston, {M}{A}: {A}rtech {H}ouse. xvi, 1992.

\bibitem{ro11}
J.~L. {R}omero.
\newblock {S}urgery of spline-type and molecular frames.
\newblock {\em {J}. {F}ourier {A}nal. {A}ppl.}, 17:135 --– 174, 2011.

\bibitem{rosh97}
A.~{R}on and Z.~{S}hen.
\newblock {W}eyl-{H}eisenberg frames and {R}iesz bases in
  ${L}_2(\mathbb{R}^d)$.
\newblock {\em {D}uke {M}ath. {J}.}, 89(2):237--282, 1997.
\bibitem{ru66}
W.~{R}udin.
\newblock {\em {R}eal and {C}omplex {A}nalysis}.
\newblock {M}c{G}raw-{H}ill {B}ook {C}ompany, {N}ew {Y}ork, 1966. 
\bibitem{suzh02}
W.~{S}un and X.~{Z}hou.
\newblock {I}rregular wavelet/{G}abor frames.
\newblock {\em {A}ppl. {C}omput. {H}armon. {A}nal.}, 13(1):63--76, 2002.

\bibitem{dohove11}
G.~A. {V}elasco, N.~{H}olighaus, M.~{D}{\"o}rfler, and T.~{G}rill.
\newblock {C}onstructing an invertible constant-{Q} transform with
  non-stationary {G}abor frames.
\newblock In {\em {P}roceedings of {D}{A}{F}{X}11, {P}aris}, September 2011.

\bibitem{wa92}
D.~F. {W}alnut.
\newblock {C}ontinuity properties of the {G}abor frame operator.
\newblock {\em {J}. {M}ath. {A}nal. {A}ppl.}, 165(2):479--504, 1992.

\bibitem{maxi01}
Z.~{X}iong and H.~{M}alvar.
\newblock {A} nonuniform modulated complex lapped transform.
\newblock {\em {I}{E}{E}{E} {S}ignal {P}rocessing {L}etters}, 8(9):257--260,
  {S}eptember 2001.

\end{thebibliography}

\end{document}